\documentclass[a4paper,11pt]{amsart}

\usepackage{hyperref}
\hypersetup{colorlinks,citecolor=blue,linkcolor=blue}

\usepackage{amscd}
\usepackage{mathptmx}
\usepackage{amsmath}
\usepackage{amssymb}
\usepackage{amsthm}
\usepackage{array}
\usepackage{enumerate}
\usepackage{graphicx}
\usepackage{lmodern}
\usepackage{mathrsfs}

\newtheorem{theorem}{Theorem}[section]
\newtheorem{theoremI}{Theorem}
\newtheorem{lemma}[theorem]{Lemma}
\newtheorem{prop}[theorem]{Proposition}
\newtheorem{cor}[theorem]{Corollary}
\newtheorem{corI}{Corollary}

\theoremstyle{definition}
\newtheorem{definition}[theorem]{Definition}

\theoremstyle{remark}
\newtheorem{rmk}[theorem]{Remark}

\numberwithin{equation}{section}

\DeclareMathOperator{\SL}{SL}
\DeclareMathOperator{\GL}{GL}

\DeclareMathOperator{\PO}{PO}

\DeclareMathOperator{\SU}{\mathbf{SU}}

\DeclareMathOperator{\U}{\mathbf{U}}
\DeclareMathOperator{\Sp}{Sp}
\DeclareMathOperator{\PSp}{PSp}
\DeclareMathOperator{\PU}{PU}

\DeclareMathOperator{\PGL}{PGL}
\DeclareMathOperator{\Aut}{Aut}

\DeclareMathOperator{\vol}{vol}

\newcommand{\J}{\mathcal{J}_k}
\newcommand{\Jo}{\J^\infty}

\newcommand{\Z}{\mathbb Z}
\newcommand{\Q}{\mathbb Q}
\newcommand{\R}{\mathbb R}
\renewcommand{\H}{\mathbb H}
\newcommand{\Hu}{\mathscr H}
\newcommand{\Ic}{\mathscr I}

\newcommand{\A}{\mathbb A}
\newcommand{\f}{\mathfrak f}

\renewcommand{\O}{\mathcal O}
\renewcommand{\o}{\mathfrak o}
\newcommand{\kun}{\widehat{k}_v}
\newcommand{\oun}{\widehat{\o}_v}

\newcommand{\oP}{\overline{P}_v}

\newcommand{\D}{\mathscr D}

\newcommand{\Af}{\A_\mathrm{f}}

\newcommand{\CP}{\mathfrak{C}_P}

\newcommand{\V}{\mathcal{V}_k}
\newcommand{\Vf}{\V^{\mathrm{f}}}

\newcommand{\Vi}{\V^\infty}
\newcommand{\Ram}{\mathcal{R}}
\newcommand{\id}{\mathrm{id}}
\newcommand{\tA}{\mathrm{A}}

\newcommand{\tC}{\mathrm{C}}

\newcommand{\matr}[4]{\left(\begin{array}{cc}#1 & #2 \\ #3 & #4 \end{array}\right)}

\newcommand{\Hy}{\mathbf H}

\newcommand{\G}{\mathbf G}
\newcommand{\aG}{\overline{\G}}

\newcommand{\CG}{\mathbf C}

\newcommand{\Gs}{\mathscr G}

\newcommand{\cMv}{\overline{\mathscr M}_{\!\! v}}
\newcommand{\Mv}{\overline{\mathrm M}_{v}}

\newcommand{\bs}{\backslash}

\newcommand{\muEP}{\mu^\mathrm{EP}}

\newcommand{\Eul}{\mathscr{E}}

\newenvironment{myarray}[1][1]{%
  \renewcommand*{\arraystretch}{#1}%
  \array%
}{%
  \endarray
}

\begin{document}
\title[Quaternionic hyperbolic lattices of minimal covolume]{Quaternionic hyperbolic lattices \\of minimal covolume}

\begin{abstract}
        For any $n>1$ we determine the uniform and nonuniform lattices of the smallest covolume in
        the Lie group $\Sp(n,1)$. We explicitly describe them in terms of the
        ring of Hurwitz integers in the nonuniform case with $n$ even, respectively, of the icosian ring
        in the uniform case for all $n>1$.
\end{abstract}

\author{Vincent Emery and Inkang Kim}
\thanks{The first author is supported by the Swiss National Science Foundation,
        Project number  PP00P2\_157583. The Second author gratefully acknowledges the partial
        support of Grant
NRF-2019R1A2C1083865 and KIAS Individual Grant (MG031408).}

\address{
Bern University of Applied Sciences\\
School of Engineering and Computer Science\\
Quellgasse 21\\
CH-2501 Bienne\\
Switzerland
}

\address{
Korea Institute for Advanced Study\\
School of Mathematics\\
85 Hoegiro, Dongdaemun-gu\\
Seoul 130-722\\
Korea
}

\email{vincent.emery@math.ch}
\email{inkang@kias.re.kr}

\date{\today}


\maketitle

\section{Introduction}
\label{sec:intro}

\subsection{The problem}

The purpose of this article is to determine the lattices in the Lie group $G =
\PSp(n,1)$ of minimal covolume, for any integer $n > 1$. For other rank one real simple Lie
groups (namely $G = \PO(n,1)$, and $G = \PU(n,1)$) this problem has been
addressed in several different papers (see for instance
\cite{GehrMart09,Hild07,BelEme,EmeSto}). The result in the case $G=\PGL_2(\R)$ is a classical
theorem of Siegel \cite{Sieg45}. Many of the results mentioned above are
restricted to the class of arithmetic lattices: this allows the use of Prasad's volume formula
\cite{Pra89} along with techniques from Borel-Prasad \cite{BorPra89} as the main ingredient in the proof, 
and we shall adopt the same strategy in this paper.  A significant advantage when treating the
case $G= \PSp(n,1)$ is that all lattices are arithmetic (since superrigidity holds), so that the results
obtained below solve the problem in this Lie group. 


It will be more convenient to work with lattices in the group $\Sp(n,1)$,
which is a double cover of $\PSp(n,1) = \Sp(n,1)/\left\{ \pm I \right\}$.
Let $\H$ denotes the Hamiltonian quaternions. By definition, $\Sp(n,1)$ is the unitary 
group $\U(V_\R, h)$ of $\H$-linear automorphisms of $V_\R = \H^{n+1}$ preserving the hermitian form
\begin{align}
        h(x, y) &=  - \overline{x_0} y_0 + \overline{x_1} y_1 + \dots +
        \overline{x_n} y_n.
        \label{eq:hermit-n-1}
\end{align}
In Sect.~\ref{sec:Psp} we explain  how our results translate 
back into the original problem in $\PSp(n,1)$, and we discuss their
geometric meaning in terms of quaternionic hyperbolic orbifolds.

We will use the Euler-Poincar\'e characteristic  $\chi$ (defined in the
sense of C.T.C.~Wall) as a measure of the covolume: there exists a normalization
$\muEP$ of the Haar measure on $\Sp(n,1)$ such that $\muEP(\Gamma\bs \Sp(n,1)) =
\chi(\Gamma)$ for any lattice $\Gamma \subset \Sp(n,1)$; see
Sect.~\ref{sec:mu-EP-mu}. The problem is then
to find the lattices $\Gamma \subset \Sp(n,1)$ with minimal value for
$\chi(\Gamma)$. It is usual (and natural) to separate the problem into the
subcases of $\Gamma$ uniform (i.e., the quotient $\Gamma\bs \Sp(n,1)$ being
compact), respectively $\Gamma$ nonuniform.

\subsection{The nonuniform case}
Denote by $\Hu \subset \H$ the ring of Hurwitz integers, which consists of
elements of the form $\alpha_0 + \alpha_1 i + \alpha_2 j +\alpha_3 k \in \H$
with either all $\alpha_i \in \Z$, or all $\alpha_i \in \Z +\frac{1}{2}$. Let
$\Sp(n,1, \Hu)$ be the subgroup $\U(L, h) \subset \U(V_\R, h)$ stabilizing the
lattice $L = \Hu^{n+1} \subset V_\R$.  In matrix notation
it corresponds to the set of  elements of $\Sp(n,1)$ with
coefficients in $\Hu$, whence the notation.  The group $\Sp(n,1, \Hu)$ is a
nonuniform lattice of $\Sp(n,1)$ (see Sect.~\ref{sec:BHC-Godement}). 
It is easily checked that it is
normalized by the scalar matrix $g= I (1+ i)/\sqrt{2}$, for which $g^2 \in
\Sp(n,1, \Hu)$ holds. We denote by $\Gamma^0_n$ the subgroup extension
of $\Sp(n,1, \Hu)$ by $g$. Thus $\Gamma^0_n$ contains $\Sp(n,1, \Hu)$ as a subgroup of index $2$. 
For $n=2$ this lattice has been considered in \cite[Prop. 5.8]{KimPar03}. We
will compute the following (see Corollary~\ref{cor:vol-Sp-Hur}):
\begin{align}
        \chi(\Gamma^0_n) = \frac{(n+1)}{2} 
        \prod_{j=1}^{n+1} \frac{2^j + (-1)^j}{4j} |B_{2j}|,
        \label{eq:vol-Gamma-Hu}
\end{align}
where $B_m$ is the $m$-th Bernoulli number.
For the reader's convenience we list the first few values for
$\chi(\Gamma^0_n)$ in Table~\ref{tab:expl}. Note the few distinct prime factors
(namely $p=2,3$) appearing for $n=2$; this compares with $\chi(\SL_2(\Z)) = -1/12$. 
\begin{theoremI}
        \label{thm:nonunif}
       For any $n$ even, the lattice $\Gamma^0_n$ realizes the smallest covolume 
       among nonuniform lattices in $\Sp(n,1)$. Up to conjugacy, it is the unique lattice
       with this property.
\end{theoremI}

At this point we would like to stress the relative simplicity of the description 
of the lattice $\Gamma^0_n$. In comparison, the results of
\cite{Belo04,BelEme,EmeSto}
concerning $\PO(n,1)$ and $\PU(n,1)$ describe the minimal covolume lattices as
normalizers of \emph{principal} arithmetic subgroups, i.e., by using a
local-to-global (adelic) description that heavily depends on Bruhat-Tits theory (see
Sect.~\ref{sec:princ-subgps-and-volumes}). A more concrete description in those cases is only available in low
dimensions (in the form of Coxeter groups) or in a few special cases (see for
instance \cite{Eme-Appendix,Eme-unimd}). Another situation where satisfactory
descriptions are available is the case of a \emph{split} Lie group $G$ (see for instance
\cite{Thilmany}, and \cite{Lub90,SG09} in the positive characteristic case).

The adelic description of arithmetic subgroups is the setting needed to 
apply Prasad's formula, and in this respect the proof of
Theorem \ref{thm:nonunif} (and Theorem \ref{thm:unif}
below) follows the same strategy as in those previous articles.
The improvement in the present case is stated in Theorem
\ref{thm:covolume-U-L-h}, where we have been able to express a large
class of stabilizers of hermitian lattices -- including $\Sp(n,1, \Hu)$ -- as principal
arithmetic subgroups, in particular permitting the computation of their covolumes. This
makes these subgroups more tractable to geometric or algebraic investigation;
for instance, the reflectivity of $\Sp(n,1, \Hu)$ has
already been studied by Allcock in \cite{Allc00}.


For $n$ odd, there is a nonuniform lattice of covolume smaller than $\Gamma_n^0$:
\begin{theoremI}
        \label{thm:nonunif-odd}
        Let $n>1$ be odd. There exists a unique (up to conjugacy) nonuniform
        lattice $\Gamma^1_n \subset \Sp(n,1)$ of minimal covolume. It is commensurable with
        $\Gamma^0_n$, and 
        \begin{align}
                \chi(\Gamma^1_n) = 
                \frac{(n+1)}{2} \prod_{j=1}^{n+1} \frac{2^{2j} - 1}{4j} |B_{2j}|
                \prod_{j=1}^{\frac{n+1}{2}} \frac{1}{2^{4j} -1}.
                \label{eq:vol-Gamma-1}
        \end{align}
\end{theoremI}
For the lattice $\Gamma_n^1$ we did not manage to find an alternative to
the construction relying on principal arithmetic subgroups. Thus a possible
improvement of our result would be to obtain a more concrete description
for it.
\begin{rmk}
        \label{rmk:Gamma-s}
        The notation has been chosen so that $\Gamma^s_n$ denotes the nonuniform 
        lattice of the \underline{s}mallest covolume for any $n$, when setting $s = (n \mod 2)
        \in \left\{ 0,1 \right\}$.
\end{rmk}

\begin{table}
        \centering
       $$ 
       \begin{myarray}[1.3]{lcccc}
                \hline
                n= & 2 & 3 & 4 & 5\\ \hline \\[-5pt]
                \chi(\Gamma^0_n) & \dfrac{1}{2^{11}\cdot 3^{3}} &
                \dfrac{17}{2^{14}\cdot 3^5\cdot 5} & \dfrac{17\cdot
                31}{2^{19}\cdot 3^{6}\cdot 11} 
                & \dfrac{17\cdot 31 \cdot 691}{2^{22}\cdot 3^{7}\cdot5\cdot
        7\cdot 11} \\[10pt] \hline \\[-5pt]
        \chi(\Gamma^1_n) &  &
                \dfrac{1}{2^{14}\cdot 3^2\cdot 5^2} & &
                \dfrac{31 \cdot 691}{2^{22}\cdot 3^{3}\cdot5^3 \cdot 7\cdot 13} 
                \\[10pt] \hline \\[-5pt]
        \chi(\Delta_n) & \dfrac{67}{2^{10}\cdot 3^{3}\cdot 5^3\cdot7} &
                \dfrac{19^2\cdot 67}{2^{13}\cdot 3^{5}\cdot 5^4\cdot7} &
                \dfrac{19^2\cdot 67\cdot191\cdot2161}{2^{18}\cdot 3^{6}\cdot
        5^5\cdot7\cdot11} &
        \\[10pt] \hline
        \end{myarray}$$
        \vspace{3pt}
        \caption{Some explicit values for $n \le 5$}
        \label{tab:expl}
\end{table}
\subsection{The uniform case}


Let $k = \Q(\sqrt{5})$, and let $\Ic$
denote the \emph{icosian ring}, i.e., $\Ic$ is the unique (up to conjugacy) 
maximal order in the quaternion $k$-algebra $\left( \frac{-1, -1}{k} \right)$ (see
\cite[Sect.~8.2]{Conway-Sloane}, or \cite[p.141]{Vign80}). We have an inclusion
$\Ic \subset \H$.
The following hermitian form resticts to the standard  $\Ic$-lattice
$\Ic^{n+1}$ in $V_\R = \H^{n+1}$:
\begin{align}
        \label{eq:h-omega}
        h(x, y) &=  \tfrac{1-\sqrt{5}}{2} \overline{x_0} y_0 + \overline{x_1} y_1 + \dots +
        \overline{x_n} y_n.
\end{align}
The stabilizer $\U(\Ic^{n+1}, h)$ is a uniform lattice in $\Sp(n,1)$ (see
Sect.~\ref{sec:BHC-Godement}), which we will denote by the symbol $\Delta_n$ in the
following.

\begin{theoremI}
        \label{thm:unif}
        For any $n > 1$, the lattice $\Delta_n$ realizes the smallest covolume 
        among uniform lattices in $\Sp(n,1)$. Up to conjugacy, it is the unique lattice
        with this property. Its Euler characteristic is given by 
        \begin{align}
                \chi(\Delta_n) = (n+1) \prod_{j=1}^{n+1} \frac{\zeta_k(1-2j)}{4}, 
                \label{eq:vol-Lambda-Ic}
       \end{align}
       where $\zeta_k$ denotes the Dedekind zeta function of
       $k=\Q(\sqrt{5})$.
\end{theoremI}

\begin{rmk}
        \label{rmk:zeta_k}
        The special values $\zeta_k(1-2j)$ appearing in \eqref{eq:vol-Lambda-Ic} are known
        to be rational by the Klingen-Siegel theorem (more generally, for any totally real $k$),
        and they can be precisely evaluated (see \cite[Sect.~3.7]{Serr71}).
        A list for $j=1,\ldots,5$ is given for instance in
        \cite[Table~2]{Eme-Euler2}, from which we obtain the explicit values for
        $\chi(\Delta_n)$
        listed in Table~\ref{tab:expl}. We omit $n=5$ for reason of space.
\end{rmk}

\subsection{Numerical values, growth}

We can now compare the nonuniform and uniform lattices, and study the asymptotic
of their covolumes with respect to the dimension.
We give a few numerical values in Table \ref{tab:both}. One sees that
$\chi(\Gamma^s_n)$ and $\chi(\Delta_n)$ are very close for $n=2$, but then  
$\chi(\Delta_n)$ starts growing much faster than $\chi(\Gamma^s_n)$ (which also
grows with $n>5$). More precisely, we can state the following result, which
essentially follows from Theorems~\ref{thm:nonunif}--\ref{thm:unif}; see 
Sect.~\ref{sec:proof-growth} for the discussion of the proof.
\begin{corI}
        Each of the sequences $\chi(\Gamma^s_n)$, $\chi(\Delta_n)$, and
        $\chi(\Delta_n)/\chi(\Gamma^s_n)$ grows super-expon\-entially as
        $n \to \infty$.
        \label{cor:growth}
\end{corI}
Geometric lower bounds -- by means of embedded balls  -- for the volume of
quaternionic hyperbolic manifolds have been obtained by Philippe in her thesis
(see \cite[Cor.~5.2]{PhilZ}), and in \cite[Cor.~5.3]{KimPar03} 
for noncompact manifolds. In contrast to Corollary~\ref{cor:growth}, these bounds
decrease fast with the dimension. 



It is clear from Corollary~\ref{cor:growth} that the proof of the next
result now follows by inspecting a finite number of values. 

\begin{corI}
        \label{cor:smallest-nonunif}
        For $n=2$ the lattice of the smallest covolume in $\Sp(n,1)$ is 
        uniform. For any $n>2$ this lattice is nonuniform. 
        The smallest Euler characteristic of a lattice in $\Sp(n,1)$ (with $n>1$
        arbitrary) is given by $\chi(\Gamma^1_5)$.
\end{corI}




\begin{table}
        \centering
       $$ 
       \begin{myarray}[1.3]{rll}
                \hline
                n & \multicolumn{1}{c}{\chi(\Gamma^s_n)} &
                \multicolumn{1}{c}{\chi(\Delta_n)} \\ \hline 
                2 & 1.808 \times 10^{-5} & 2.769 \times 10^{-6} \\
                3 & 2.712 \times 10^{-7} & 2.777 \times 10^{-6} \\
                4 & 1.253 \times 10^{-7}  & 2.171 \times 10^{-4} \\
                5 & 1.662 \times 10^{-8} & 3.162   \\
                10 & 1.736 \times 10^8 & 5.771 \times 10^{64} \\
                15 & 8.624 \times 10^{55} & 3.510 \times 10^{218} \\
                20 & 1.654 \times 10^{151}  & 1.833 \times 10^{478} \\
                \hline
        \end{myarray}$$
        \vspace{3pt}
        \caption{Some approximate values}
        \label{tab:both}
\end{table}

\subsection{Quaternionic hyperbolic orbifolds}
\label{sec:Psp}

        Let $\pi: \Sp(n,1) \to \PSp(n,1)$ denote the projection. 
        Any lattice $\Gamma' \subset \PSp(n,1)$ is the image of a lattice $\Gamma =
        \pi^{-1}(\Gamma')$ that contains the center $\left\{ \pm I \right\}$. Then we
        have $\chi(\Gamma') = 2 \chi(\Gamma)$. It follows that $\Gamma'$ is of
        minimal covolume in $\PSp(n,1)$ if and only if so is $\Gamma$ in $\Sp(n,1)$
        (note that a lattice of minimal covolume in $\Sp(n,1)$ necessarily
        contains the center $\left\{ \pm I \right\}$).

        The group $\PSp(n,1)$ identifies with the 
        isometries of the quaternionic hyperbolic $n$-space $\Hy_\H^n$.
        For any lattice $\Gamma' \subset \PSp(n,1)$ we consider the finite-volume
        quaternionic hyperbolic orbifold  $M = \Gamma'\bs\Hy_\H^n$.
        Alternatively, we may write $M$ as the quotient $M = \Gamma\bs\Hy_\H^n$,
        where $\Gamma = \pi^{-1}(\Gamma')$. Then the \emph{orbifold}
        Euler-Poincar\'e characteristic of $M$ is given by
        $\chi(M) = \chi(\Gamma') = 2 \chi(\Gamma)$. In case $\Gamma'$ is
        torsion-free, $M$ is a quaternionic hyperbolic \emph{manifold}, and $\chi(M)$
        corresponds to the usual (i.e., topological) Euler characteristic. The volume of the orbifold $M$ is
        proportional to $\chi(M)$ (see below). Thus,
        Theorems~\ref{thm:nonunif}--\ref{thm:unif} determine the quaternionic 
        hyperbolic orbifolds (compact or noncompact) of the smallest volume.

        The choice of a normalization of the volume form on $\Hy^n_\H$ induces
        a volume form on its compact dual, i.e., on the quaternionic projective space
        $\H P^n$. For the induced volume on a quotient $M = \Gamma'\bs\Hy^n_\H$
        we have:
        \begin{align}
                \vol(M) = \frac{\vol(\H P^n)}{n+1} \chi(M),
                \label{eq:vol-chi-M}
        \end{align}
        where $n+1$ appears as the Euler characteristic of $\H P^n$.

\subsection{Outline}
The classification of arithmetic subgroups in $\Sp(n,1)$ is discussed in
Sect.~\ref{sec:arithm-sbgp}. In Sect.~\ref{sec:parah-Galois-cohom} we
recall some materials from Bruhat-Tits theory, in particular to prepare the
discussion of Prasad's volume formula in
Sect.~\ref{sec:princ-subgps-and-volumes}.
Sect.~\ref{sec:stabilizers-lattices} deals with lattices that are defined
as stabilizers of hermitian modules. The proofs of the results stated in
the introduction are contained in Sect.~\ref{sec:maximal-sbgps}, with the
exception of the uniqueness, which is proved in Sect.~\ref{sec:unique}.

\subsection*{Acknowledgement} 
We would like to thank Jeff Meyer for helpful correspondence concerning the
classification of arithmetic subgroups in $\Sp(n,1)$, and Ruben B\"ar for his
help with the icosian ring. The first author thanks the KIAS for hospitality and 
financial support during two short stays in Korea, and the HIM in Bonn and the
organizers of the program ``Periods in Number Theory, Algebraic Geometry and
Physics'' for providing a nice working atmosphere. 
We also thank the referee for a careful reading and many useful remarks.

\section{Arithmetic subgroups in $\Sp(n,1)$}
\label{sec:arithm-sbgp}

\subsection{Admissible groups}
\label{admissible-groups}

Let $k$ be a number field, and consider
an absolutely simple algebraic $k$-group $\G$ such that
\begin{align}
        \G(k \otimes_\Q \R) \cong \Sp(n,1) \times K
        \label{eq:admissibility-cond}
\end{align}
for some compact group $K$. 
Then $\G$ is simply connected of type $\tC_{n+1}$, and $k$ is totally real.
Moreover, we can fix an embedding $k \subset \R$ so that $\G(\R)$ identifies with
$\Sp(n,1)$. 
It follows from the classification of simple algebraic group (see \cite{Tits66}
and Remark~\ref{rmk:SU-U} below) 
that $\G$ is isomorphic  
        over $k$  to some unitary group $\U(V, h)$,  where
\begin{itemize}
        \item $D$ is a quaternion algebra over $k$,
                with the standard involution $x \mapsto \overline{x}$.
        \item $V$ is a right $D$-vector space; 
        \item $h$ is a nondegenerate hermitian form on $V$ (sesquilinear with respect to
                the standard involution).
\end{itemize}
Such a $k$-group $\G = \U(V, h)$ satisfying \eqref{eq:admissibility-cond} will be called
\emph{admissible}, and in this case we shall use the same terminology for 
the hermitian space $(V, h)$. We call $D$ the \emph{defining algebra} of $\G$.
Facts concerning quaternion algebras will be recalled along the lines; we refer
to \cite{Vign80} or \cite[Ch.~2 and 6]{MaclReid03}.

\begin{rmk}
       For any field extension $K/k$ we have by definition $\G(K) = \U(V_K, h)$, where
       $V_K = V \otimes_k K$. The latter is seen as a $D_K$-module, for $D_K = D
       \otimes_k K$. In particular, we can use the notation
       $\U(V_k, h)$ to denote the $k$-points $\G(k)$.   
\end{rmk}

\begin{rmk}
        \label{rmk:SU-U} 
        In \cite[p.~56]{Tits66} Tits describes the classification in terms of the \emph{special} unitary
        group $\SU$, however, in our case $\SU = \U$ since symplectic transformations have
        determinant~$1$.
\end{rmk}


\subsection{Admissible defining algebras}
\label{sec:condition-on-D-h-unique}
We denote by $\V = \Vi \cup \Vf$ the set of (infinite or finite) places of $k$,
and, for any $v \in \V$, by $D_v$  the quaternion algebra $D_{k_v} = D \otimes_k k_v$. The
algebra $D$ is completely determined by the set of places $v \in \V$ where it
ramifies (i.e., for which $D_v$ is a division algebra), and the set of such places
is of even (finite) cardinality. There is no other obstruction to the
existence of a quaternion algebra $D$ with prescribed localizations
$\left\{ D_v \, |\, v \in V \right\}$; see \cite[Sect.~7.3]{MaclReid03}.

Let $\G = \U(V, h)$ as above, with $V$ over $D$. The following (well-known)
result appears as a special case of Lemma~\ref{lemma:G(R)-split} below. 
Recall that $D_v$ is said \emph{split} if it is isomorphic to $M_2(k_v)$, and
this happens exactly when $D_v$ is not ramified. 
\begin{lemma}
        \label{lemma:G-splits-D-splits}
        Let $v \in \V$. The algebraic group $\G_{k_v}$ (obtained by scalars extension)
        splits if and only if $D_v$ splits (i.e., is not ramified).
\end{lemma}
\begin{proof}
        It follows from the classification in \cite[p.~56]{Tits66} that
        $\G_{k_v}$ has relative rank less than $n+1$ if $D_v$ is a division
        algebra; thus in this case $\G_{k_v}$ is not split.
        If $D_v$ splits the fact that $\G_{k_v}$ splits will follow from
        Lemma~\ref{lemma:G(R)-split} below with $R = k_v$. 
\end{proof}

\begin{cor}
        \label{cor:G-determines-D}
       The $k$-isomorphism class of $\G$ determines its defining algebra $D$
       uniquely up to $k$-isomorphism.
\end{cor}
\begin{proof}
       This is now clear, since $D$ is determined by the set of places where it
       splits.
\end{proof}

\begin{cor}
        \label{cor:D-admissible}
       In order for the hermitian space $(V, h)$ over $D$ to be
       admissible, it is necessary that $D$ ramifies at any $v \in \Vi$, i.e.,
       $D_v \cong \H$ for any $v \in \Vi$.
\end{cor}
\begin{proof}
       By the admissibility condition, for any $v \in \Vi$ the group $\G(k_v)$ is either 
       $\Sp(n,1)$ or $\Sp(n+1)=\Sp(n+1,0)$. For $n>1$ these groups are not split. 
       It follows from Lemma \ref{lemma:G-splits-D-splits} that $D$ ramifies at each $v \in \Vi$, so that $D_v \cong \H$
       (see \cite[Sect.~2.5]{MaclReid03}).
\end{proof}

A pair $(k, D)$ with $k \subset \R$ a totally real number field and $D$ a quaternion
algebra over $k$ will be called \emph{admissible} if $D$ satisfies the necessary  
condition of Corollary~\ref{cor:D-admissible}. More simply, we say that ``$D$ is
admissible''.

\begin{prop}
        \label{prop:D-determines-h}
       Let $(V, h)$ and $(V', h')$ be two admissible hermitian spaces of the same dimension
       over the same quaternion $k$-algebra $D$. Then $\U(V, h)$ is $k$-isomorphic to
       $\U(V', h')$. 
\end{prop}
\begin{proof}
        Being admissible, the two hermitian spaces $(V,h)$ and $(V',h')$ have the same signature over
        $k_v$ for any $v \in \Vi$, and it follows from
        \cite[10.1.8~(iii)]{Schar85} that $(V, h) \cong (V', h')$.
\end{proof}

\begin{rmk}
        \label{rmk:bijection-commens-classes}
        There is actually a bijection between the set of admissible pairs
        $(k, D)$ for $k\subset\R$ totally real, and the set of algebraic groups
        that are admissible for $\Sp(n,1)$; see \cite[Sect.~4]{Meyer}. 
        We will not need this fact in its full generality.
\end{rmk}

\subsection{The classification of lattices} 
\label{sec:BHC-Godement}
Let $\G$ be an admissible $k$-group, and let $\O_k$ denote the
ring of integers in $k$. By the Theorem of Borel and Harish-Chandra, any
subgroup $\Gamma \subset \G(\R)$  that is commensurable with
$\G(\O_k)$ (for some embedding $\G \subset \GL_N$) is a lattice in $\Sp(n,1)$.
Such a subgroup is called \emph{arithmetic}.  
The work of Margulis \cite{Ma}  has shown that superrigidity for $\Sp(n,1)$ 
(later proved by Gromov and Schoen \cite{GS} in the nonarchimedean case, and Corlette \cite{Co}   in 
the archimedean) 
implies the arithmeticity of any lattice in $\Sp(n,1)$, i.e., 
any lattice in $\Sp(n,1)$ can be constructed as an arithmetic subgroup, as
above. A pair $(k, \G)$ with $\G$ admissible determines exactly one
commensurability classes of lattices in $\Sp(n,1)$ (up to conjugacy); see
\cite[Prop.~2.5]{PraRap09}.  We will say that the lattices $\Gamma$ in such a commensurability
classes are \emph{defined over $k$}. Moreover, with Corollary~\ref{cor:G-determines-D} 
we see that the defining $k$-algebra $D$ of $\G$ is an invariant of the
commensurability class. We take over the terminology, and say that that $D$ is
the \emph{defining algebra} of $\Gamma$. Conversely, by Proposition
\ref{prop:D-determines-h}, the pair $(k,D)$ uniquely determines the
commensurability class. 

\begin{prop}[Compactness criterion]
       A lattice $\Gamma \subset \Sp(n,1)$ is nonuniform if and only if it is
       defined over $\Q$.
\end{prop}
\begin{proof}
  This is specialization of Godement's compactness criterion, which asserts that
  an arithmetic subgroup of $\G$ semisimple is nonuniform in $\G(k\otimes_\Q
  \R)$ if and only if $\G$ is $k$-isotropic. 
  If $k \neq \Q$, an admissible $k$-group $\G$ has a compact factor
  $\G(k_v)$ for some $v \in \Vi$, so that $\G$ cannot be isotropic. Let $k=\Q$,
  and let $\G = \U(V, h)$ admissible defined over $\Q$. Then $\G$ is isotropic
  when $(V, h)$ is, and by \cite[Theorem 10.1.1]{Schar85} this happens exactly
  when its trace form $q_h$ (which is a quadratic form over $\Q$ in $4(n+1)$
  variables) is isotropic. Then  $\G$ is isotropic by \cite[Cor.~5.7.3 (iii)]{Schar85}.
\end{proof}

We will describe in Sect.~\ref{sec:stabilizers-lattices} a concrete way to
construct arithmetic subgroups in $\Sp(n,1)$.

\section{Parahoric subgroups and Galois cohomology}
\label{sec:parah-Galois-cohom}

We collect in this section some notions of Bruhat-Tits theory; we refer to \cite{Tits79}.
Sect.~\ref{sec:parah-sbgps} and \ref{sec:parah-gp-scheme} are needed for the
volume computation in Sect.~\ref{sec:princ-subgps-and-volumes}. The content of
Sect.~\ref{sec:Galois-conj-action} will appear later, in Sect.~\ref{lemma:index}, and
the reader might want to skip it until they reach this point.

In this section $\G$ denotes an admissible $k$-group. For any finite place $v
\in \Vf$ we denote by $\o_v$ the valuation ring in $k_v$. 

\subsection{Parahoric subgroups}
\label{sec:parah-sbgps}

For any $v \in \Vf$ 
the group $\G(k_v)$ acts on its associated Bruhat-Tits building $X_v$. 
A \emph{parahoric subgroup} $P_v \subset \G(k_v)$ is
by definition a stabilizer of a facet of $X_v$ (note that we are working with
$\G$ simply connected). Maximal parahoric subgroups correspond to stabilizers of
vertices on $X_v$. If $\Delta_v$ denotes the affine root system of $\G(k_v)$,
then the conjugacy classes of parahoric subgroups $P_v \subset \G(k_v)$ are in
correspondence with the subsets $\Theta_v \subsetneq \Delta_v$; then $\Theta_v$
is called the \emph{type} of $P_v$. The correspondence preserves the inclusion,
and thus  maximal subgroups have types that omit exactly one element in $\Delta_v$. 

Assume first that  $\G_{k_v}$ is split. Then its affine root
system is given by the following local Dynkin diagram:
\begin{align}
        \includegraphics[width=6cm]{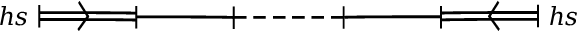}
        \label{eq:local-dynkin-split}
\end{align}
The parahoric subgroups of maximal volume in $\G(k_v)$ are those that
are \emph{hyperspecial}, i.e., of type $\Delta_v \smallsetminus \left\{\alpha 
\right\}$ where $\alpha$ is any of the two hyperspecial vertices (labelled `hs' in
\eqref{eq:local-dynkin-split}). 

There is exactly one non-split form of $\G_{k_v}$ of type $\tC_{n+1}$, and it
splits over the maximal unramified extension $\kun/k_v$. If $\G_{k_v}$
is not split, it corresponds to the local type named $^2\tC_{n+1}$ in 
\cite[Sect.~4.3]{Tits79}; we have reproduced in Table~\ref{tab:local-Dynkin} the corresponding
local indices (their description depends on the parity of $n$).
It is a general fact, proved in \cite[App.~A]{BorPra89}, that parahoric subgroups of maximal volume
in $\G(k_v)$ are those of type $\Delta_v \smallsetminus \left\{ \alpha
\right\}$ where $\alpha$ is \emph{very special} (see \emph{loc.\ cit.}\ for the
definition).
In our case, the very special vertex in $\Delta_v$ is $\alpha_0$
for $n$ even, resp. $\alpha_1$ for $n$ odd (as shown in
Table~\ref{tab:local-Dynkin}; note that $\alpha_1$ is defined only for $n$ odd).
Thus, for $s = n \mod 2$, a parahoric subgroup of $\G(k_v)$ is of maximal 
volume exactly when it is of type 
$\Theta_v = \Delta_v \smallsetminus \left\{ \alpha_s \right\}$. 

\begin{table}
        \centering
        \begin{tabular}{ccc}
                $n=2m$ & & $n=2m-1$
                \\
        \includegraphics[width=5.5cm]{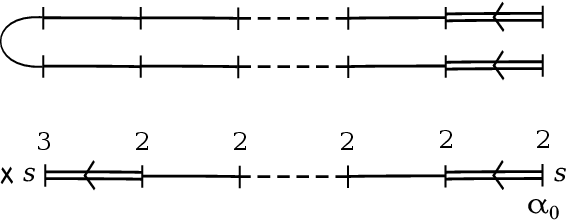} & &
        \includegraphics[width=5.5cm]{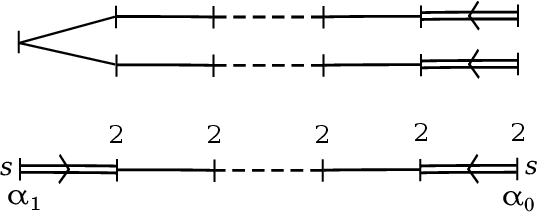}\\[0.5cm] 
        \end{tabular}
        \caption{Local indices for the type $^2\tC_{n+1}$}
        \label{tab:local-Dynkin}
\end{table}

\subsection{The group scheme structure}
\label{sec:parah-gp-scheme}
Let $\G(k_v)$ as above, split or not. Each parahoric subgroup $P_v \subset
\G(k_v)$ can be written as $P_v = \Gs(\o_v)$ for some canonical smooth group
scheme $\Gs$ over $\o_v$. Let $\f_v$ be the residue field of $k_v$.
For a fixed $P_v$, following the notation of \cite{Pra89}, we denote 
by $\Mv$ the maximal reductive quotient of the $\f_v$-group $\Gs_{f_v}$
obtained from $\Gs$ by base change; $\Mv$ is connected. The structure of $\Mv$
can be obtained from the type of $P_v$ by using the procedure explained in
\cite[Sect.~3.5.2]{Tits79}. 
If $\G_{k_v}$ is split, then $P_v$ is hyperspecial if and
only if $\Mv$ is simple of type $\tC_{n+1}$ (i.e., of the same type as
$\G_{k_v}$). In this case we will write $\Mv = \cMv$.


\subsection{The Galois cohomology action}
\label{sec:Galois-conj-action}
For $\G$ admissible over $k$, we let $\CG$ denote its center and $\aG = \G/\CG$ its 
adjoint quotient. For any field extension $K/k$, the group $\aG(K)$ identifies
with the group of inner $K$-automorphisms of $\G$. We denote by $\delta$ the
connecting map in the Galois cohomology exact sequence:
\begin{align}
        \label{eq:delta}
        1 \to \CG(K) \to \G(K) \stackrel{\pi}{\to} \aG(K) \stackrel{\delta}{\to} H^1(K, \CG) \to
        H^1(K, \G).
\end{align}
For $K = k_v$ nonarchimedean this provides an action of $H^1(k_v, \CG)$ on the
local Dynkin diagram $\Delta_v$ (see \cite[Sect.~2.8]{BorPra89}); the action respects the symmetries of
$\Delta_v$, so that $H^1(k_v, \CG) \to \Aut(\Delta_v)$. Note in particular that
for $\G_{k_v}$ nonsplit of type $\tC_{n+1}$, we have $\Aut(\Delta_v) = 1$. In
the split case, there is exactly one nontrivial symmetry.  
We denote by $\xi$ the induced ``global'' map $H^1(k, \CG) \to \prod_{v\in\Vf}
\Aut(\Delta_v)$ (the image actually lies in the direct product).

Of particular interest to us is the subgroup $\delta(\aG(k))' = \delta(\aG(k) \cap \pi(\G(\R)))$ 
(the notation follows \cite[Sect.~2.8]{BorPra89}; recall that we have fixed an
inclusion $k \subset \R$). We will make use of the following alternative
description.
\begin{lemma}
        \label{lemma:A-dGk}
        The group $\delta(\aG(k))'$ coincides with the kernel of the diagonal
        map 
        \begin{align*}
                H^1(k, \CG) \to \prod_{v \in \Vi} H^1(k_v, \CG).
        \end{align*}
\end{lemma}
\begin{proof}
        For $x \in H^1(k, \CG)$ we denote by $(x_v)_{v\in\Vi}$ its image in
        $\prod_{v\in\Vi} H^1(k_v, \CG)$.  Let us first assume $x \in \delta(\aG(k))'$.
        Then for the place $v = \id$ corresponding to the inclusion $k \subset \R$ we 
        have  $x_v \in (\delta \circ \pi)(\G(\R))$, which by exactness of
        \eqref{eq:delta} (with $K = \R$) is equivalent to $x_v = 1$.
        For $v \neq \id$, the group $\G(k_v)$ is compact and in this case it is
        known that $\pi: \G(k_v) \to \aG(k_v)$ is surjective (see
        \cite[Sect.~3.2: Cor.~3]{PlaRap94}). Since $x_v \in
        \delta(\aG(k_v))$, we thus have $x_v = 1$ again by the exactness of
        \eqref{eq:delta}. Conversely, suppose that $x$ has trivial image in
        $\prod_{v\in\Vi}H^1(k_v, \CG)$. Then $(x_v)$ has certainly trivial image in 
        $\prod_{v \in \Vi} H^1(k_v, \G)$. But for $\G$ simply connected the
        latter identifies with $H^1(k, \G)$ by the Hasse principle (see
        \cite[Theorem 6.6]{PlaRap94}).     
        By exactness of \eqref{eq:delta} with $K = k$, it follows that $x \in
        \delta(\aG(k))'$.
\end{proof}
\begin{rmk}
        We will see in Lemma \ref{lemma:Sp-surj-PSp} that actually $\G(k_v) \to
        \aG(k_v)$ is surjective for $v = \id$ as well, which slightly simplifies
        the proof. We have preferred giving the above proof, which works quite generally
        when $\G$ is simply connected; it appears in \cite[Sect.~12.2]{EmePhD}. 
\end{rmk}

\section{Principal arithmetic subgroups and volumes}
\label{sec:princ-subgps-and-volumes}

\subsection{Principal arithmetic subgroups}
\label{sec:principal-sbgp}
For $\G$ an admissible  $k$-group, we will denote by $P = (P_v)_{v \in \Vf}$ a collection
of parahoric subgroups $P_v \subset \G(k_v)$ ($v \in \Vf$). Such a collection is called
\emph{coherent} if the product $\prod_{v\in \Vf} P_v$ is open in the group $\G(\Af)$,
where $\Af$ denotes the finite ad\`eles of $k$. This condition implies that $P_v$ is
hyperspecial for all but finitely many $v \in \Vf$. Moreover, one has that the subgroup 
$\Lambda_P = \G(k) \cap \prod_{v \in \Vf} P_v$ is an arithmetic subgroup of $\G(k)$,
called \emph{principal}.

\subsection{The normalized Haar measure}
\label{sec:mu-EP-mu}

The covolume of the principal arithmetic subgroup $\Lambda_P \subset \G(k)$ can be computed with Prasad's
volume formula \cite{Pra89}, in terms of the volumes of the parahoric subgroups $P_v$ ($v\in
\Vf$). In the notation of {\em loc.\ cit.}, our situation corresponds to the case $\G_S =
\G(\R)$ (i.e., $S$ contains a single infinite place corresponding to the inclusion $k
\subset \R$). We write  $\mu = \mu_S$ for  the
normalization of the Haar measure on $\G(\R)$ used in \cite[Sect.~3.6]{Pra89}. Then for the Euler-Poincar\'e
characteristic (in the sense of C.T.C.\ Wall) of $\Gamma \subset \G(\R)$ 
one has $|\chi(\Gamma)| = |\chi(X_u)|\, \mu(\Gamma\bs\G(\R))$,
where $X_u$ is the compact dual symmetric space associated with $\Hy^n_\H$ (see \cite[\S
4]{BorPra89}). Explicitly, $X_u =
\Sp(n+1)/(\Sp(n) \times \Sp(1))$ is the quaternionic projective space $\H P^n$, for which
$\chi(X_u) = n+1$. Moreover, since the symmetric space of
$\Sp(n,1)$ has dimension $4n$, it follows from \cite[Prop.~23]{Serr71} that
$\chi(\Gamma)$ is positive. Thus 

\begin{align}
        \chi(\Gamma) &=  (n+1) \cdot \mu(\Gamma\bs \G(\R)). 
       \label{eq}
\end{align}


\subsection{Prasad's volume formula}
\label{sec:prasad}
To state the volume formula for $\Lambda_P \subset \G(k)$ in an explicit way we need to introduce
some notation; we mostly follow \cite{Pra89}. The symbol $\D_k$ denotes the absolute value of the
discriminant of $k$, and we write $d = [k:\Q]$ for the degree.  
For each $v \in \Vf$ let $\f_v$ be the residue field of $k_v$, and let  $q_v$ be the
cardinality of $\f_v$. For each parahoric subgroup $P_v$ the reductive
$\f_v$-group $\Mv$ is defined in Sect.~\ref{sec:parah-gp-scheme}. The reductive
$\f_v$-group corresponding to a hyperspecial parahoric subgroup in the split form of $\G$ is
denoted by $\cMv$. For all but finitely many $v \in \Vf$ we have that $P_v$ is
hyperspecial and thus $\Mv \cong \cMv$.  In our situation, 
Prasad's volume formula \cite[Theorem~3.7]{Pra89}  takes the
following form (note that in our case $\ell = k$ since $\G$ is of type $\tC$ and thus has
no outer symmetries):
\begin{align}
        \mu(\Lambda_P\bs \G(\R)) &= \D_k^{\dim \G /2} \left( \prod_{j = 1}^{n+1}
        \frac{(2j-1)!}{(2 \pi)^{2j}}  \right)^d \Eul(P) 
        \label{eq:Prasad-1}
\end{align}
where the ``Euler product'' $\Eul(P)$ is given by
\begin{align}
        \Eul(P) &= \prod_{v\in\Vf} \frac{q_v^{(\dim \Mv + \dim
        \cMv)/2}}{|\Mv(\f_v)|} .  
        \label{eq:Eul-P}
\end{align}

\subsection{The Euler product and zeta functions}
\label{sec:lambda-factors}

Let $T$ be the finite set of places $v \in \Vf$ such that $P_v$ is not
hyperspecial. For $v \not\in T$ we have that $\Mv \cong \cMv$, which is an $\f_v$-split simple
group of type $\tC_{n+1}$, for which $|\cMv(\f_v)| = q_v^{(n+1)^2}
\prod_{j=1}^{n+1}(q_v^{2j} -1)$ (see \cite[Tab.~1]{Ono66}), and $\dim \cMv = \dim \G = (n+1)(2n+3)$.  
By a direct computation we may rewrite $\Eul(P)$ as the following:
\begin{align}
        \Eul(P) &= \prod_{v \in T} e'(P_v) \prod_{v \in \Vf}  \nonumber
        \frac{q_v^{\dim \cMv}}{|\cMv(\f_v)|}\\
        &= \prod_{v \in T} e'(P_v) \prod_{v \in \Vf} \prod_{j=1}^{n+1}
        \frac{1}{1 - q_v^{-2j}} \nonumber \\ 
        &= \prod_{v \in T} e'(P_v) \prod_{j=1}^{n+1} \zeta_k(2j), 
        \label{eq:Eul-ep}
\end{align}
where $\zeta_k$ is the Dedekind zeta function of $k$, and the correcting factors
$e'(P_v)$ (so called ``lambda factors'' in \cite{BelEme}) are given by 
\begin{align}
        e'(P_v) &= q_v^{(\dim \Mv - \dim \cMv)/2} \frac{|\cMv(\f_v)|}{|\Mv(\f_v)|}.
        \label{eq:ePv}
\end{align}
Putting together  Equations \eqref{eq}, \eqref{eq:Prasad-1} and
\eqref{eq:Eul-ep} we can finally write (where the second line is obtained 
from the functional equation of $\zeta_k$; see \cite[Ch.VII~(5.11)]{Neuk99}):
\begin{align}
        \label{eq:Prasad-zeta}
       \chi(\Lambda_P)  &= (n+1) \D_k^{\dim \G /2} \prod_{v \in T} e'(P_v)
        \prod_{j=1}^{n+1} \left(\frac{(2j-1)!}{(2 \pi)^{2j}}\right)^d \zeta_k(2j)\\
        &= (n+1) \prod_{v \in T} e'(P_v) \prod_{j=1}^{n+1} 2^{-d} |\zeta_k(1-2j)|.
        \label{eq:Prasad-final}
\end{align}

\subsection{The nonsplit local factors}
\label{sec:ePv-not-split}

We compute in the following lemma the local factors $e'(P_v)$ of interest to us.

\begin{lemma}
        \label{lemma:vol-eP-max}
        Suppose that $\G_{k_v}$ is nonsplit, and let $P^t_v \subset \G(k_v)$
        be a special parahoric subgroup of type $\Delta_v \smallsetminus \left\{
        \alpha_t \right\}$ (assuming $n$ odd if $t=1$). Then
        \begin{align}
                \label{eq:eP0}
                e'(P^0_v) &= \prod_{j=1}^{n+1} (q_v^j + (-1)^j); \\
                e'(P^1_v) &=  \frac{\prod_{j=1}^{2m}
                (q_v^{2j}-1)}{\prod_{j=1}^m(q_v^{4j} -1)},
                \label{eq:eP1}
        \end{align}
        where $n+1 = 2m$ in the latter.
\end{lemma}
\begin{rmk}
        Note that $e'(P_v^1)$ is clearly an integer. The fact that $P^1_v$ has
        larger volume than $P^0_v$ (see Sect.~\ref{sec:parah-sbgps}) is reflected by 
        the fact that $e'(P^1_v)$ is smaller than $e'(P^0_v)$; this inequality
        can be checked empirically
        (and probably rigorously with some effort) from the formulas in Lemma \ref{lemma:vol-eP-max}.
        One may notice that, as polynomials in $q_v$, these two local factors have quite
        similar order of magnitude though. 
\end{rmk}

\begin{proof}
        The definition of $e'(P^t_v)$ is given in \eqref{eq:ePv}. 
        The dimension and order for $\cMv$ are given in
        Sect.~\ref{sec:lambda-factors}. 
        We obtain the description of $\Mv$ (see Sect.~\ref{sec:parah-gp-scheme}) in each case by
        \cite[Sec.~3.5.2]{Tits79}.  We refer to \cite[Tab.~1]{Ono66} for
        the order of the classical finite simple groups. For $P_v^0$  we have that
        $\Mv$ is given as an almost direct product $\Mv = T \cdot H$, where $T$ is a
        nonsplit one dimensional torus, and $H$ is simple of type $^2\tA_n$.
        In particular, $\dim \Mv = (n+1)^2$, and $|T(\f_v)| =q_v+1$. By Lang's
        isogeny theorem we have $|\Mv(\f_v)| = (q_v + 1) |H(\f_v)|$ and the
        formula for $e'(P_v^0)$ now follows from a straightforward computation. 
        Note that for $n = 2$ the local index $^2\tC_3$ needs to be listed
        separately (see \cite[p.63]{Tits79}); however the description for $\Mv$
        is similar, and the formula remains the same. 

        For $P_v^1$ with $n+1 = 2m$, we have that $\Mv$ is obtained by Weil
        restriction of scalars as $\Mv = \mathrm{Res}_{\mathfrak{k}_v/\f_v}(H)$, where
        $H$ is simple of type $\tC_m$ and $\mathfrak{k}_v/\f_v$ is quadratic
        (i.e., $\mathfrak{k}_v$ has cardinality $q_v^2$) .
        Thus  $\Mv$ has twice the dimension of $H$, which is $m(2m+1)$, and 
        $|\Mv(\f_v)| = |H(\mathfrak{k}_v)|$. The result for $e'(P_v^1)$ follows
        directly.
\end{proof}

\section{Stabilizers of hermitian lattices}
\label{sec:stabilizers-lattices}

In this section we obtain the covolume of the lattices $\Gamma^0_n$ and $\Delta_n$
in $\Sp(n,1)$ as a consequence of Theorem~\ref{sec:volume-stabilizer}. To prove
the latter, we first need to study the structure of the stabilizers of hermitian
lattices; this is done in Sect.~\ref{sec:O-D-R} and
\ref{sec:Lemma-Scharlau-conseq}. In those sections $R$ will denote an 
integral domain containing the ring of integers $\O_k$, and $K$ will be the field 
of fractions of $R$.

\subsection{Hermitian lattices over orders}
\label{sec:O-D-R}
Let us fix an order $\O_D$ in an admissible quaternion $k$-algebra $D$, 
and consider the right $\O_D$-module $L = \O_D^{n+1}$. We set $\O_{D, R}  = \O_D
\otimes_{\O_k} R$. Then $L_R = L \otimes_{\O_k} R$ is a
right $\O_{D,R}$-module.
Consider a hermitian form $h$ on $L$, described as follows:
\begin{align}
        h(x, y) &= \sum_{i=0}^n a_i \overline{x}_i y_i, 
        \label{eq:hermitian-form-module}
\end{align}
where the coefficients $a_i$ are taken in $\O_k$.  Note that the standard involution on
$D$ restricts to $\O_D$ (use the trace), so that $(L, h)$ is indeed a hermitian module in
the sense of \cite[Ch.~7]{Schar85} (and so is $(L_R, h)$ for any ring extension $R$ of $\O_k$).  
We write $V = L_k = D^{n+1}$, and we will assume that $(V, h)$ is admissible. 
We say in this case that the module $(L, h)$ itself is \emph{admissible}.

Recall that by definition the hermitian module $(L_R, h)$ is \emph{regular} (or
\emph{nonsingular}) if the map $\phi_h: x \mapsto h(x, \cdot)$ induces an isomorphism of
$\O_{D,R}$-modules from $L_R$ onto its dual module $(L_R)^*$, seen as a right
module via $f \alpha = \overline{\alpha} f$ (see \cite[Sect.~7.1]{Schar85}).
When $R$ is a field this
is equivalent to $(L_R, h)$ being nondegenerate, i.e., $(L_R)^\perp =
0$. For more general $R$ we will need the following result.

\begin{lemma}
       If the coefficients of $h$ are invertible in $R$ (i.e., $a_i \in R^\times$ for $i =
       0,\dots, n$) then the hermitian module $(L_R, h)$ is regular.
        \label{lemma:ai-L_R-regular}
\end{lemma}
\begin{proof}
        Let $\left\{ \mathbf{e}_i \right\}$ be the standard basis of $L_R = \O_{D,R}^{n+1}$, 
        and let $\left\{ \mathbf{e}_i^* \right\} \subset (L_R)^*$ be the
        associated dual basis. We have $\phi_h(\mathbf{e}_i) =
        a_i \mathbf{e}^*_i = \mathbf{e}^*_i a_i$ (note that $a_i =
        \overline{a}_i$, since $a_i \in \O_k$). The map $\mathbf{e}_i^* \mapsto
        \mathbf{e}_i a_i^{-1}$ from $(L_R)^*$ to $L_R$ is then inverse to $\phi_h$.
        \end{proof}

\subsection{A key lemma}
\label{sec:Lemma-Scharlau-conseq}

For $(L, h)$ and $(V, h)$ as above, 
consider the stabilizer of $L$ in $\G(k) = \U(V, h)$, i.e., the subgroup
\begin{align}
        \U(L, h) &= \left\{ g \in \U(V, h) \,|\, g L = L \right\}.
        \label{eq:SU-L-h}
\end{align}
This is an arithmetic subgroup of $\G(k)$.
More generally, we will denote by $\U(L_R, h) \subset \G(K)$ the stabilizer of $L_R \subset
V_K$. The following lemma is the key result that will be used in Sect.~\ref{sec:Stab-local-global}.

\begin{lemma}
        \label{lemma:G(R)-split}
        Assume that the following conditions hold:
        \begin{enumerate}
                \item $R$ is a principal ideal domain; 
                \item $\O_{D,R}$ splits, i.e., $\O_{D,R} \cong M_2(R)$;
                \item the hermitian module $(L_R, h)$ is regular.
        \end{enumerate}
        Then there is an isomorphism $\phi: \U(V_K, h) \to \Sp_{2n+2}(K)$ such
        that $\phi(\U(L_R, h)) = \Sp_{2n+2}(R)$.
\end{lemma}
\begin{proof}
      We adapt the discussion from \cite[pp.~361-362]{Schar85} (which considers
      \emph{skew}-hermitian spaces, only over fields) to our setting. First we
      may fix an identification $\O_{D, R} = M_2(R)$; the standard involution 
      is then given by 
      \begin{align}
              \overline{\matr{a}{b}{c}{d}} &= \matr{d}{-b}{-c}{a}.
              \label{eq:stand-invol-M2}
      \end{align}
      Let $e_1 = \matr{1}{0}{0}{0}$  and $e_2 = \matr{0}{0}{0}{1}$, so
      that $V_K$ has the following splitting: $V_K = V_K e_1 \oplus V_K e_2$. We
      set $V_1 =  V_K e_1$. As for {\em loc.\ cit.} we obtain from $h$ a bilinear
      form $b_h$ on $V_1$ determined by 
      \begin{align}
              h(x e_1, y e_1) = \matr{0}{0}{b_h(x e_1, y e_1)}{0},
              \label{eq:b_h-from-h}
      \end{align}
      and in our case $b_h$ is easily seen to be antisymmetric. Since $(L_R, h)$
      is a hermitian module, the form $b_h$ actually restricts to a bilinear 
      (antisymmetric) form on the  $R$-lattice $L_1 = L_R e_1$ of $V_1$.  Note
      that $L_1$ is free over $R$ of rank $2n+2$. If $f \in L_1^*$ then we can 
      extend $f$ to $L_R$ by setting for any $x \in L_R$: 
      \begin{align}
              \tilde{f}(x) = f(x e_1) + f(x e e_1) e,
              \label{eq:f-tilde}
      \end{align}
      where $e = \matr{0}{1}{1}{0}$. One computes that this is indeed an
      extension, which is actually $\O_{D,R}$-linear, i.e., $\tilde{f} \in (L_R)^*$.
      In particular, we have that
      the symplectic module $(L_1, b_h)$ is regular, as $(L_R, h)$ itself is
      assumed to be regular. Since by assumption $R$ is a PID, we can now deduce
      that $(L_1, h)$ is a orthogonal sum of hyperbolic modules (see for
      instance \cite[Prop.~2.1]{KirkMcDo81}), and thus its isometry group is
      isomorphic to $\Sp_{2n+2}(R)$.

      An analogous formula to \eqref{eq:f-tilde} can be used to extend any isometry $\sigma$
      of $(L_1, h)$ to an isometry $\tilde{\sigma} \in \U(L_R, h)$ (see
      \cite[p.362]{Schar85}). This shows
      that $g \mapsto g|_{L_1}$ yields an isomorphism from $\U(L_R, h)$ to
      $\Sp_{2n+2}(R)$. The same construction with $R = K$ thus provides the
      isomorphism $\phi$ in the statement.
\end{proof}




\subsection{The local structure of lattice stabilizers}
\label{sec:Stab-local-global}
Let again $(L, h)$ denote an admissible lattice over $\O_D$,
with $D$ defined over the number field $k$ and $\O_D \subset D$ an order.
The following (non-standard) terminology will be convenient for us.
\begin{definition}
        \label{def:max-type}
        We say that $(L, h)$ is \emph{of maximal type} 
        if $\O_D$ is maximal and $(L, h)$ is regular. 
\end{definition}
\begin{rmk}
        \label{rmk:exist-max-type}
       Given $D$, the existence of an admissible $(L, h)$ of maximal type does not seem to be
       obvious (and we believe that it is wrong in general).
\end{rmk}

For each finite place $v \in \Vf$ we 
shall abbreviate the notation from Sect.~\ref{sec:O-D-R} (with $R = \o_v$) as follows:
$L_v = L_{\o_v}$. As above, $\Ram$ denotes the set of finite places 
$v \in \Vf$ where $D_v$ ramifies, and $\G = \U(V, h)$.

\begin{lemma}
        \label{lemma:stab-hypersp}
        Assume that $(L, h)$ is of maximal type, and let $v \in \Vf$ be a finite
        place with $v \notin \Ram$. Then  $\U(L_v, h)$ is a hyperspecial
        parahoric subgroup in $\G(k_v) \cong \Sp_{2n+2}(k_v)$. 
\end{lemma}
\begin{proof}
        The order $\O_{D, \o_v}$, being maximal in $D_v \cong M_2(k_v)$, 
        must be conjugate to $M_2(\o_v)$ (see \cite[Ch.~6]{MaclReid03}). Thus we can apply
        Lemma~\ref{lemma:G(R)-split}: it implies that  $\U(L_v, h)$ identifies with   
        $\Sp_{2n+2}(\o_v)$, which is hyperspecial parahoric by \cite[Sect.~3.4.2]{Tits79}.
\end{proof}

We now turn our attention to the  case of places where $D$ ramifies.

\begin{lemma}
       \label{lemma:stab-spec}
       For $(L, h)$ of maximal type and $v \in \Ram$, the subgroup 
       $\U(L_v, h)$ is a special parahoric subgroup in $\G(k_v)$ of type
       $\Delta_v \smallsetminus \left\{ \alpha_0 \right\}$.
\end{lemma}
\begin{proof}
        Let $\kun$ be the maximal unramified extension of $k_v$, with ring of integers
        $\oun$.  
        Let  $\O_{D,v} = \O_{D, \o_v}$ and $\widehat{\O} = \O_{D,v} \otimes_{\o_v} \oun$. 
        The latter in an order in $\widehat{D}_v = D_v \otimes \kun$, and we consider a maximal 
        order $\O' \subset \widehat{D}_v$ containing $\widehat{\O}$. That is,
        \begin{align}
                \label{eq:Ov}
                \O_{D,v} \subset \widehat{\O} \subset \O'\,.
        \end{align}
        Thus $\O' \cap D_v$ is an order in $D_v$, which equals $\O_{D,v}$ since the latter is maximal. 
        Note that $\widehat{D}_v$ is split (see \cite[Theorem 2.6.5]{MaclReid03}).

        From the inclusions~\eqref{eq:Ov} we may interpret the subgroup $P_v = \U(L_v, h) \cong \U(\O_{D,v}^{n+1}, h)$ 
        of $\G(k_v)$ as a subgroup of the matrix group $P_v' = \U((\O')^{n+1}, h)$. The latter is a
        hyperspecial parahoric of $\G(\kun)$ by Lemma~\ref{lemma:stab-hypersp}. From the equality
        $\O' \cap D_v = \O_{D,v}$ we deduce $P_v' \cap \G(k_v) = P_v$. But in view of 
        the local indices in Table~\ref{tab:local-Dynkin} this means that $P_v$ is a special parahoric
        subgroup of type $\Delta_v \smallsetminus \left\{ \alpha_0 \right\}$
        (since $\alpha_0$ is the
        unique affine root in $\Delta_v$ that appears as the restriction of 
        hyperspecial roots of $\G(\kun)$).
\end{proof}

\subsection{The volume formula for the maximal type}
\label{sec:volume-stabilizer}
Lattices of maximal type are particularly interesting because of the
following result. Recall that $q_v$ denotes the cardinality of the residue field
of $k_v$ (for $v \in \Vf$). See Definition~\ref{def:max-type} for ``maximal type''. 
\begin{theorem}
        Let $(L, h)$ be an admissible hermitian $\O_D$-lattice of maximal type. Then 
        $\U(L, h)$ is a principal arithmetic subgroup of $\G(k) = \U(V, h)$; and  
        \begin{align}
                \chi(\U(L, h)) &= (n+1) \prod_{j=1}^{n+1} \left(
                \frac{\zeta_k(1-2j)}{2^{\scriptscriptstyle [k:\Q]}} \prod_{v \in \Ram} q_v^j + (-1)^j \right),
                \label{eq:covolume-U-L-h}
        \end{align}
        where $\Ram$ is the set of finite places where $D$ ramifies. 
        \label{thm:covolume-U-L-h}
\end{theorem}
\begin{proof}
        We can write $L = V_k \cap \prod_{v \in \Vf} L_v$, from which
        we obtain (for $\G(k)$ diagonally embedded in $\prod_v \G(k_v)$): 
        \begin{align*}
                 \U(L, h) &= \G(k) \cap \prod_{v \in \Vf} \U(L_v, h).
        \end{align*}
        With Lemmas \ref{lemma:stab-hypersp} and \ref{lemma:stab-spec} this shows
        that $\U(L, h)$ is principal, and the formula for $\chi(\U(L, h))$ is
        deduced from \eqref{eq:Prasad-final} and Lemma \ref{lemma:vol-eP-max}. 
\end{proof}

We emphasize the special case of the standard hermitian form over the
Hurwitz integers in the next corollary. It implies the formula in
\eqref{eq:vol-Gamma-Hu}, since we have 
$\chi(\Gamma^0_n) = \chi(\Sp(n,1, \Hu)) / 2$ by construction.

\begin{cor}
        \label{cor:vol-Sp-Hur}
       The arithmetic subgroup $\Sp(n,1, \Hu)$ is principal, and 
       \begin{align}
               \chi(\Sp(n,1, \Hu)) &= (n+1) \prod_{j=1}^{n+1} \frac{2^j + (-1)^j}{4j}
               |B_{2j}|,
               \label{eq:vol-SpHu}
       \end{align}
       where $B_m$ is the $m$-th Bernoulli number.
\end{cor}
\begin{proof}
        We have that $\Hu$ is a maximal order in $D = \Hu \otimes \Q = \left(
        \frac{-1, -1}{\Q} \right)$, and
       the latter is the quaternion $\Q$-algebra that ramifies exactly at $p = 2$ 
       and $p = \infty$ (see \cite[p.79]{Vign80}). By Lemma~\ref{lemma:ai-L_R-regular} it is clear that 
       $(\Hu^{n+1}, h)$, with $h$ given in \eqref{eq:hermit-n-1}, is of maximal type.
       Thus we can apply the theorem, and the formula in \eqref{eq:vol-SpHu} follows immediately 
       from the known expression:
       \begin{align*}
               \zeta(-m) &= (-1)^m \frac{B_{m+1}}{m+1}.
       \end{align*}
\end{proof}

\subsection{The covolume of $\Delta_n$}
\label{sec:L-h-over-5}

Let $L = \Ic^{n+1}$, where $\Ic$ is the icosian ring. The hermitian form
\eqref{eq:h-omega} has been chosen so that $(L, h)$ is of maximal type. 
By definition, $\Delta_n = \U(L, h)$.
The formula in \eqref{eq:vol-Lambda-Ic} for $\chi(\Delta_n)$ is thus an
immediate consequence of Theorem~\ref{thm:covolume-U-L-h}, since in this case
$\Ram = \emptyset$ (see \cite[p.150]{Vign80}).

\section{The minimality of $\chi(\Gamma^s_n)$ and $\chi(\Delta_n)$}
\label{sec:maximal-sbgps}

\subsection{Normalizers of minimal covolume}
Let $\Gamma \subset \Sp(n,1)$ be a \emph{maximal lattice}, i.e., maximal with respect 
to  inclusion as in Section \ref{sec:BHC-Godement}.  We have $\Gamma \subset \G(\overline{k} \cap \R)$ for some
admissible $k$-group $\G$, but the stricter inclusion $\Gamma \subset \G(k)$ does not
hold in general.  By \cite[Prop.~1.4]{BorPra89} we have that $\Gamma$ is a normalizer 
$N_{\G(\R)}(\Lambda_P)$, where $\Lambda_P \subset \G(k)$ is a principal
arithmetic subgroup. 

Let $P = (P_v)$ be a coherent collection such that any parahoric subgroup $P_v$
is of maximal volume. Then $\Lambda_P$ is of minimal covolume among arithmetic
lattices contained in $\G(k)$. It is a priori not clear -- but turns out to be
true -- that the normalizer $N_{\G(\R)}(\Lambda_P)$ for such a choice of $P$ is
of minimal covolume in its commensurability class in $\G(\R)$. This can be
proved in the same way as in \cite[Sect.~4.3]{BelEme} (see also 
\cite[Sect.~12.3]{EmePhD} for a more detailed exposition). We state the result in
the following lemma. In the rest of this section $\Ram$ denotes the set of
finite places where $\G$ does not split (equivalently, where its defining
algebra ramifies). 

\begin{lemma}
        \label{lemma:descript-local-minimal}
        The lattice $\Gamma = N_{\G(\R)}(\Lambda_P)$ is of
        minimal covolume in its commensurability class if and only if the
        coherent collection $P= (P_v)$ satisfies 
        \begin{enumerate}
                \item $P_v$ is hyperspecial for each $v\notin \Ram$; and
                \item $P_v$ is special of maximal volume for $v \in \Ram$.
        \end{enumerate}
\end{lemma}

Recall that for $v \in \Ram$ a parahoric subgroup $P_v \subset \G(k_v)$ is
special of maximal volume exactly when it has type $\Delta_v \smallsetminus
\left\{ \alpha_s \right\}$, where $s = n \mod 2$. 
Using  Lemmas \ref{lemma:stab-hypersp} and \ref{lemma:stab-spec} we thus have:

\begin{cor}
        \label{cor:max-type-min-covol}
        Let $(L, h)$ be an admissible $\O_D$-lattice of maximal type. If $n$ is
        even or $\Ram =\emptyset$,  
        then the lattice $N_{\Sp(n,1)}(\U(L, h))$ is of minimal covolume in its
        commensurability class.
\end{cor}

\subsection{The index computation}
We will need to estimate the index $[\Gamma : \Lambda_P]$ for 
$\Gamma = N_{\G(\R)}(\Lambda_P)$ of minimal covolume in its commensurability
class. For this we state the following lemma, which considers a slightly more general
situation. The symbol $h_k$ denotes the class number of $k$, and $U_k$ (resp.\ $U_k^+$) 
are the units (resp.\ totally positive units) in $\O_k$. 
\begin{lemma}
        \label{lemma:index}
        Let $P = (P_v)$ such that $P_v$ is hyperspecial for any $v \in \Vf
        \smallsetminus \Ram$. Then
       \begin{align*}
               [\Gamma : \Lambda_P] &\le 2^{\#\Ram} \cdot h_k \cdot |U^+_{k}/U_k^2|.
       \end{align*}
\end{lemma}
\begin{proof}
        We assume the notation of Sect.~\ref{sec:Galois-conj-action}; in
        particular, $\CG$ is the center of $\G$. We set $A = \delta(\aG(k))'$.
        Let $\Theta = (\Theta_v)_{v\in\Vf}$ be the type of the coherent collection $P$.
        From the assumption it follows that none of the
        types $\Theta_v$ has symmetries, and thus the stabilizer of $\Theta$ in
        $A$ equals the kernel $A_\xi$ of $\xi$.
        By \cite[Prop.~2.9]{BorPra89} we thus have the exact sequence 
        \begin{align}
                \label{eq:Rohlfs-ext-sq}
                1 \to \CG(\R)/(\CG(k) \cap \Lambda_P) \to \Gamma/\Lambda_P \to A_\xi \to 1.
        \end{align}
        In our case $\CG = \mu_2$, and it follows that the left part vanishes. Hence,
        $[\Gamma:\Lambda_P] = |A_\xi|$. We now use the identification $H^1(k, \mu_2) =
        k^\times/(k^\times)^2$. For $S \subset \Vf$ any finite set of places we
        define
        \begin{align*}
                k_{2, S} = \left\{\left. x \in k^\times \; \right| \; v(x) \in
                2\Z \;\; \forall v \in \Vf \smallsetminus S \right\},
        \end{align*}
        and $k_2 = k_{2, \emptyset}$. Since $\Aut(\Delta_v)$ is trivial for
        any $v \in \Ram$, it follows from \cite[Prop.~2.7]{BorPra89} that
        $H^1(k, \CG)_\xi = k_{2, \Ram}/(k^\times)^2$. See \cite{BorPra89} for the definitions. For $k_{2, S}^+ \subset
        k_{2, S}$ denoting the subgroup consisting of totally positive elements, we
        conclude from Lemma~\ref{lemma:A-dGk} that $A_\xi = k_{2,
        \Ram}^+/(k^\times)^2$. This group contains $k_2^+/(k^\times)^2$
        with index at most $2^{\#\Ram}$, and the order of the latter can be bounded by $h_k \cdot
        |U^+_{k}/U_k^2|$ by using the same argument as in the proof of \cite[Prop.~0.12]{BorPra89}.
\end{proof}

The following is obtained as a corollary of the proof.

\begin{cor}
        \label{cor:index-Q} 
        Let $k = \Q$ and $\#\Ram = 1$, and assume $P$ as above. Then $[\Gamma :
        \Lambda_P] = 2$.
\end{cor}
\begin{proof}
        Let $\Ram = \left\{ p \right\}$, with $p>0$. Lemma~\ref{lemma:index} shows
        $[\Gamma:\Lambda_P] \le 2$, but on the other hand $p$ provides a
        nontrivial element in $\Q^+_{2, \Ram}/(\Q^\times)^2$. 
\end{proof}

\subsection{The nonuniform lattices $\Gamma_n^s$}


By definition $\Gamma^0_n = \left< g, \Sp(n,1, \Hu) \right>$ with $g$ of order $2$ that
normalizes $\Sp(n,1, \Hu)$. Let $D$ be the defining algebra of $\Gamma_n^0$
(i.e., $D = \Hu \otimes \Q$). It ramifies precisely at $p=2$ and $p=\infty$.
Thus we can apply Corollary~\ref{cor:index-Q}, and it follows that $\Gamma_n^0$
coincides with the normalizer of $\Sp(n,1, \Hu)$ in $\Sp(n,1)$. For $n$ even, 
Corollary \ref{cor:max-type-min-covol} then shows that $\Gamma^0_n$ is of minimal 
covolume in its commensurability class. On the other hand, Rohlfs' criterion
\cite[Satz~3.5]{Rohlfs79} shows that $\Gamma_n^0$ is maximal (w.r.t.\ inclusion)
for any $n>1$.

For $n>1$ odd, we construct $\Gamma_n^1$ as follows. Let $\G$ be the $\Q$-group
that contains $\Sp(n,1, \Hu)$. We choose a coherent collection $P = (P_v)$ with
$P_v$ hyperspecial for $v \neq 2$, and $P_v = P^1_v$ of type $\Delta_v
\smallsetminus \left\{ \alpha_1 \right\}$ for $v = 2$. Let $\Gamma_n^1 =
N_{\G(\R)}(\Lambda_P)$. By Lemma~\ref{lemma:descript-local-minimal} it is of
minimal covolume in its commensurability class. Moreover, by
Corollary~\ref{cor:index-Q} we have $[\Gamma_n^1:\Lambda_P] = 2$. In particular,
\begin{align}
        \label{eq:vol-Gamma-1-eP}
        \chi(\Gamma_n^1) &= \frac{e'(P^1_2)}{e'(P^0_2)} \chi(\Gamma_n^0),
\end{align}
from which we obtain the formula in \eqref{eq:vol-Gamma-1} with
Lemma~\ref{lemma:vol-eP-max}. 


The following proposition now implies -- up to the uniqueness --
Theorems \ref{thm:nonunif} and \ref{thm:nonunif-odd}. 
\begin{prop}
        \label{prop:min-nonunif}
       Let $\Gamma \subset \Sp(n,1)$ be a nonuniform lattice of minimal
       covolume, and let $s= (n \mod 2)$. Then $\Gamma$ is commensurable to $\Gamma^s_n$, and they have the same
       covolume.
\end{prop}
\begin{proof}
        We have seen that $\Gamma^s_n$ is of minimal covolume in its
        commensurability class, so  it suffices to prove the commensurability. By Sect.~\ref{sec:arithm-sbgp}, 
        $\Gamma$ nonuniform is constructed as an
        arithmetic subgroup of $\G = \U(V, h)$ for $(V, h)$ admissible, and
        $V$ a vector space over a quaternion $\Q$-algebra $D$.
        Let $\Ram$ be the set of places $v$ such that $D_v$ ramifies.
        Being of minimal covolume, we may write $\Gamma = N_{\G(\R)}(\Lambda_P)$, with $P_v$
        hyperspecial unless $v \in \Ram$ (by Lemma~\ref{lemma:descript-local-minimal}).
        Moreover, for $v \in \Ram$ the subgroup $P_v$ is of maximal volume and
        thus of type $\Delta \smallsetminus \left\{ \alpha_s \right\}$ by
        Sect.~\ref{sec:parah-sbgps}. 
        By Lemma~\ref{lemma:index} (with $k=\Q$) we have $[\Gamma:\Lambda_P] \le 2^{\#\Ram}$. 
        Together with Equation~\eqref{eq:Prasad-final} this gives: 
        \begin{align}
                \label{eq:up-bound-nonunif}
                \chi(\Gamma) &\ge
                \frac{\chi(\Lambda_P)}{[\Gamma : \Lambda_P]} \nonumber\\ 
                &\ge (n+1) \prod_{v \in \Ram} \frac{e'(P_v)}{2}
                \prod_{j=1}^{n+1} \frac{\zeta(1-2j)}{2}. 
        \end{align}
        Only the middle factor in \eqref{eq:up-bound-nonunif} depends on the
        choice of $\Lambda_P$, and  it takes the smallest possible value for $\Ram = \left\{ 2 \right\}$
        (note that $\Ram = \emptyset$ cannot appear here; see
        Sect.~\ref{sec:condition-on-D-h-unique}). But in that case this lower bound is precisely $\chi(\Gamma^s_n)$,
        whence $\chi(\Gamma) = \chi(\Gamma^s_n)$ (by minimality of $\chi(\Gamma))$.
        Since $D$ is now ramified exactly at $p = 2$ and $p=\infty$, it has same
        defining algebra as $\Gamma^0_n$ and the commensurability follows from
        Sect.~\ref{sec:BHC-Godement}.
\end{proof}

\subsection{The minimality of $\chi(\Delta_n)$} 
\label{sec:minim-Delta}

We now discuss the uniform case. Recall that the covolume of $\Delta_n$ has been discussed in
Sect.~\ref{sec:L-h-over-5}. Lemma \ref{lemma:index} (with $k =
\Q(\sqrt{5})$ and $\Ram = \emptyset$) implies that $\Delta_n$ coincides with its
own normalizer in $\Sp(n,1)$. Corollary~\ref{cor:max-type-min-covol} thus
implies that $\Delta_n$ is of minimal covolume in its commensurability class.
The following proposition proves the first statement in Theorem \ref{thm:unif}. 
In the proof we omit details that should be clear from the proof of Proposition~\ref{prop:min-nonunif}.

\begin{prop}
       \label{prop:min-unif}
       Let $\Gamma \subset \Sp(n,1)$ be a uniform lattice of minimal covolume. 
       Then $\Gamma$ is commensurable to $\Delta_n$, and they have the same
       covolume.
\end{prop}
\begin{proof}
        Let $\Gamma$ be a uniform lattice of minimal covolume, and let $k$ be its field of
        definition. Then $k$ is totally real, of degree $d \ge 2$.  Assume first
        that $k= \Q(\sqrt{5})$, and let $D$ be the defining algebra of $\Gamma$.
        It is clear that any nontrivial local factor $e'(P_v)$ would
        contribute to increase the volume formula, and this shows that $D$ does
        not ramify at any finite place; this implies that $\Gamma$ is
        commensurable to $\Delta_n$. Thus, it suffices to prove that $k =
        \Q(\sqrt{5})$, i.e., that $d = 2$ and $\D_k = 5$ (recall that $\D_k$
        denotes the discriminant in absolute value). 

        Let $\G$ be the algebraic $k$-group used to construct the arithmetic subgroup $\Gamma$, and let
        us write $\Gamma = N_{\G(\R)}(\Lambda_P)$ with $P = (P_v)$ a
        coherent collection of parahoric subgroups $P_v \subset \G(k_v)$ (of
        maximal volume).
        Combining \eqref{eq:Prasad-zeta} and Lemma~\ref{lemma:index} we find:
        \begin{align*}
                \chi(\Gamma) &\ge \frac{(n+1) \D_k^{\dim\G/2}}{2^{\#\Ram}\;
        h_k \; |U_k^+/U_k^2|} \; C(n)^d \; \prod_{v \in \Ram} e'(P_v)
        \prod_{j=1}^{n+1} \zeta_k(2j),
        \end{align*}
        with
        \begin{align}
                C(n) &= \prod_{j=1}^{n+1} \frac{(2j-1)!}{(2\pi)^{2j}}.
                \label{eq:C-n}
        \end{align}
        We clearly have $\zeta_k(2j) > 1$, and $|U_k^+/U_k^2| \le 2^{d-1}$ (by
        Dirichlet's unit theorem). Moreover, $e'(P_v) > 2$ for any $v \in \Ram$,
        so that the factor $2^{-\#\Ram}$ is compensated by the product of those
        local factors. We can use the bound $h_k \le 16 \left( \pi/12 \right)^d
        \D_k$ (see \cite[Sect. 7.2]{BelEme}: the argument given there for a non-totally
        real field $\ell$ provides the same bound for $k$; see  \cite[Sect.~15.2]{EmePhD} for details). This gives
        \begin{align*}
                \chi(\Gamma) &\ge \frac{(n+1)}{16\; \D_k\; 2^{d-1}}\;
                \left( \frac{12}{\pi} \right)^d
                \D_k^{\dim\G/2} \; C(n)^d.
        \end{align*}
        On the other hand we have 
        \begin{align*}
                \chi(\Delta_n) &= (n+1)\, 5^{\dim\G/2}\, C(n)^2 \prod_{j=1}^{n+1} \zeta_{\Q(\sqrt{5})}(2j) \\
                &< 1.2 \cdot (n+1)\, 5^{\dim\G/2}\, C(n)^2.
        \end{align*}
        Here we bound the product of zeta functions by the value $1.2$ by adapting the proof in \cite[p.760: proof of ($*$)]{Belo04}
        as follows (see \emph{loc.\ cit.} for details):
        \begin{align*}
                \prod_{j=1}^{n+1} \zeta_{\Q(\sqrt{5})}(2j) &< \zeta_{\Q(\sqrt{5})}(2)\, \zeta_{\Q(\sqrt{5})}(4)\, \zeta_{\Q(\sqrt{5})}(6)\,
                \prod_{j=4}^\infty \left(1 + \frac{2}{2^{2j}}\right)^2 \\
                                                           &< \zeta_{\Q(\sqrt{5})}(2) \,\zeta_{\Q(\sqrt{5})}(4) \,\zeta_{\Q(\sqrt{5})}(6) \,e^{1/48},
        \end{align*}
        and we evaluate this last bound with \textsf{Pari/GP} \cite{PARI2}.

        For the quotient this gives
        \begin{align}
                \label{eq:chi-Gamma-Delta-1}
                \frac{\chi(\Gamma)}{\chi(\Delta_n)} &> \frac{1}{39\, \D_k}
                \left( \frac{12}{\pi} \right)^d
                        \left( \frac{\D_k}{5} \right)^{\frac{(n+1)(2n+3)}{2}}
                        \left(\frac{C(n)}{2}  \right)^{d-2}.
        \end{align}
        Let us write $f(n, d, \D_k)$ for the bound on the right hand side. We
        have to show that $f(n, d, \D_k) \ge 1$ unless $d=2$ and $\D_k=5$. 

        The constant $C(n)/2$ is larger than $1$ for $n\ge 13$, and grows
        monotone from that point. Thus $f(n, d, \D_k) \ge f(13, d, \D_k)$, and
        it suffices to consider the range $n \in \left\{ 2, \dots, 13 \right\}$.
        For $d=2$ the smallest discriminant after $5$ is $\D_k = 8$, and
        numerical evaluation shows that $f(n, 2, 8) > 1$ for all $n \le 13$. This shows
        $d > 2$. For $k$ totally real of degree $d = 3$ the lowest
        discriminant is $\D_k = 49$. Again we check that $f(n, 3, 49) > 1$.
        And similarly with $d = 4$ and $\D_k \ge 725$. 
       
        It remains to exclude $d \ge 5$. In that case we use the
        following bound due to Odlyzko (see \cite[Tab.~4]{Odl90}) : $\D_k > (6.5)^d$.
        Then Equation~\eqref{eq:chi-Gamma-Delta-1} transforms into:
        \begin{align}
                \label{eq:chi-Gamma-Delta-2}
                \frac{\chi(\Gamma)}{\chi(\Delta_n)} &> \frac{1}{39 \cdot 5}
                \left( \frac{12}{\pi} \right)^2
                        \left( \frac{(6.5)^2}{5} \right)^{\delta(n)}
                        a(n)^{d-2},
        \end{align}
        where $a(n) = (6/\pi) C(n) (6.5)^{\delta(n)}$ and $\delta(n) =
        \dim\G/2 -1$. We check that $a(n) > 1$
        for all $n \in \left\{ 2, \dots, 13 \right\}$. The product preceeding
        $a(n)$ in \eqref{eq:chi-Gamma-Delta-2} is also easily seen to be (much) larger
        than $1$. This finishes the proof. 
\end{proof}

\subsection{The proof of Corollary \ref{cor:growth}}
\label{sec:proof-growth}
We have that each of $\chi(\Delta_n)$, $\chi(\Gamma^s_n)$, and their quotients
contains a factor $C(n)$ (which is given in \eqref{eq:C-n}); see
Equation~\eqref{eq:Prasad-zeta}. For $n$ large enough it is easily seen that this
factor grows faster than (say) $(2n+1)!$, i.e., it grows super-exponentially. This implies
immediately that $\chi(\Delta_n)$ and $\chi(\Gamma^s_n)$ grow
super-exponentially, as their remaining factors also increase with $n$. 
Moreover, the other factors appearing in $\chi(\Gamma^s_n)$ grow at most exponentially, so
that the $\chi(\Delta_n)/\chi(\Gamma^s_n)$ has a super-exponentially growth as
well.

\section{Proof of the uniqueness}
\label{sec:unique}

In this section we complete the proof of Theorems
\ref{thm:nonunif}-\ref{thm:unif}, by showing the uniqueness statements.

\subsection{The surjectivity of the adjoint map}
\label{sec:surjec-G-R)}

We start by proving the following auxiliary result.

\begin{lemma}
        \label{lemma:Sp-surj-PSp}
       For $n>1$ and $\G$ admissible for $\Sp(n,1)$, the map $\pi : \G(\R) \to
       \aG(\R)$ is surjective. 
\end{lemma}
\begin{proof}
        We have to show that  $\delta: \aG(\R) \to H^1(\R, \CG)$ has trivial
        image; see \eqref{eq:delta}. Recall that $H^1(\R, \CG) =
        \R^\times/(\R^\times)^2$. Let $\G = \U(V, h)$. 
        By \cite[Prop. 12.20 and Sect.~31.A]{BookInvol} the image of $\delta$
        corresponds (modulo squares) to elements $\alpha \in \R^\times$  such that $(V, \alpha h)$ 
        is isomorphic to $(V, h)$. For $h$ of signature $(n,1)$ with $n>1$ this
        requires $\alpha > 0$, whence the result. 
\end{proof}


\subsection{Counting the conjugacy classes}

Let $\Gamma \subset \Sp(n,1)$ be a nonuniform (resp. uniform) lattice that realizes the smallest  covolume.
Then by Sect.~\ref{sec:maximal-sbgps} we have $\Gamma =
N_{\Sp(n,1)}(\Lambda_P)$, where  $P = (P_v)$ is a coherent collection of parahoric 
subgroups  $P_v \subset \G(k_v)$ of maximal volume for each $v \in \Vf$, and
$\G$ is precisely the admissible $k$-group that determines $\Gamma^s_n$ (resp.
$\Delta_n$). For $v \in \Ram$ this determines the type $\Theta_v$ of $P_v$
uniquely (see Lemma~\ref{lemma:descript-local-minimal}), and for $v \notin \Ram$ we have that $\Theta_v$ is one of the two
conjugate hyperspecial types. These two hyperspecial types are conjugate by
$\aG(k_v)$. 
Up to $\aG(k)$-conjugacy the number of  principal
arithmetic subgroups $\Lambda_P$ with such a type $\Theta = (\Theta_v)$ is given by   
the order of following \emph{class group} (see for instance \cite[Sect.~6.2]{BelEme}):
\begin{align}
        \CP = \frac{\prod'_{v\in\Vf} H^1(k_v, \CG)}{\delta(\aG(k)) \prod_{v \in
        \Vf} \delta(\oP)},
        \label{eq:CP}
\end{align}
where $\oP \subset \aG(k_v)$ is the stabilizer of $P_v$, and in the numerator
$\prod'$ denotes the restricted product with respect to the
collection of subgroups $\delta(\oP)$.

By Lemma \ref{lemma:Sp-surj-PSp} this order also gives an upper bound on the number of
$\G(\R)$-conjugacy classes of $\Lambda_P$ (note that $k \subset \R$). Thus the
uniqueness in Theorems \ref{thm:nonunif}--\ref{thm:unif} follows immediately from the
following.

\begin{prop}
        For $k= \Q(\sqrt{5})$  (resp. $k = \Q$),
        we have  $\CP = 1$.
\end{prop}
\begin{proof}
        We have $H^1(k_v, \CG) = k_v^\times/(k_v^\times)^2$. 
        The subgroup $\delta(\oP) \subset H^1(k_v, \CG)$ equals the type
        stabilizer $H^1(k_v, \CG)_{\Theta_v}$. If $v \in \Ram$ then
        $\Aut(\Delta_v) = 1$  (see Table~\ref{tab:local-Dynkin}), so that this
        stabilizer is trivially the whole $H^1(k_v, \CG)$. 
        For $v \notin \Ram$ we have $\delta(\oP) = \o^\times_v
        (k_v^\times)^2/(k_v^\times)^2$ by \cite[Prop.~2.7]{BorPra89}. Thus, $\CP$ is
        a quotient of 
        \begin{align}
                \label{eq:CP-prime}
               \CP' &= \frac{\prod'_{v\in\Vf}k_v^\times/(k_v^\times)^2}{\delta(\aG(k)) \prod_{v \in \Vf} \o^\times_v
               (k_v^\times)^2/(k_v^\times)^2}.
        \end{align}
        Now by Lemma~\ref{lemma:Sp-surj-PSp} and the proof of Lemma~\ref{lemma:A-dGk}  we have that
        $\delta(\aG(k)) = \delta(\aG(k))' = k^{(+)}/(k^\times)^2$,
        where 
        \begin{align}
                k^{(+)}=\left\{x \in k^\times \;|\; x_v > 0 \quad \forall\, v \in
                        \Vi \right\}.
        \end{align}
        From \eqref{eq:CP-prime} we obtain an isomorphism 
        \begin{align*}
                \CP' &\cong \J/\left(k^{(+)} \Jo \J^2 \right),
        \end{align*}
        where $\J$ is the group of finite id\`eles of $k$, and $\Jo \subset \J$
        its subgroup consisting of  integral id\`eles.
        For both $k = \Q$ and $k = \Q(\sqrt{5})$ the unit group $U_k$ contains a representative 
        of each class of $k^\times/k^{(+)}$. Thus  
        \begin{align*}
                k^{(+)} \Jo &= k^{(+)} U_k \Jo \\
                          &= k^\times \Jo,
        \end{align*}
        so that
        \begin{align*}
                \CP' &\cong \J/\left(k^\times \Jo \J^2 \right).
        \end{align*}
        But the latter is a quotient of the class group of $k$ (see
        \cite[Sect.~1.2.1]{PlaRap94}), which is trivial here. 
\end{proof}

\bibliographystyle{amsplain}
\bibliography{vol-Sp.bbl}


\end{document}